\documentclass[12pt]{amsart}
\usepackage{amssymb,amsmath,amscd,enumerate,verbatim}
\usepackage[all]{xy}
\usepackage{fullpage}
\usepackage{pdfsync}
\usepackage[pagebackref]{hyperref}

\numberwithin{equation}{section}

\newtheorem{theorem}{Theorem}[section]
\newtheorem{lemma}[theorem]{Lemma}
\newtheorem{proposition}[theorem]{Proposition}
\newtheorem{corollary}[theorem]{Corollary}
\newtheorem{conjecture}[theorem]{Conjecture}

\theoremstyle{definition}

\newtheorem{def-prop}[theorem]{Definition-Proposition}

\newtheorem{remark}[theorem]{Remark}
\newtheorem{example}[theorem]{Example}

\newtheorem*{acknowledgement}{Acknowledgement}

\newtheorem{question}[theorem]{Question}

\DeclareMathOperator{\depth}{depth}

\DeclareMathOperator{\Ass}{Ass}
\DeclareMathOperator{\Min}{Min}

\DeclareMathOperator{\pd}{pd}

\newcommand{\ZZ}{{\mathbb Z}}
\newcommand{\NN}{{\mathbb N}}
\newcommand{\QQ}{{\mathbb Q}}

\def\mm{{\mathfrak m}}

\def\pp{{\mathfrak p}}
\def\qq{{\mathfrak q}}

\begin{document}

\title{Depth functions of powers of homogeneous ideals}

\author{Huy T\`ai H\`a}
\address{Tulane University \\ Department of Mathematics \\
6823 St. Charles Ave. \\ New Orleans, LA 70118, USA}
\email{tha@tulane.edu}

\author{Hop Dang Nguyen}
\address{Institute of Mathematics \\ Vietnam Academy of Science and Technology, 18 Hoang Quoc Viet \\ Hanoi, Vietnam}
\email{ngdhop@gmail.com}

\author{Ngo Viet Trung}
\address{International Centre for Research and Postgraduate Training, Institute of Mathematics \\ Vietnam Academy of Science and Technology, 18 Hoang Quoc Viet \\ Hanoi, Vietnam}
\email{nvtrung@math.ac.vn}

\author{Tran Nam Trung}
\address{Institute of Mathematics\\ Vietnam Academy of Science and Technology, 18 Hoang Quoc Viet, and TIMAS, Thang Long University, Nghiem Xuan Yem road, Hanoi, Vietnam}
\email{tntrung@math.ac.vn}

		
\begin{abstract}
We settle a conjecture of Herzog and Hibi, which states that the function $\depth S/Q^n$, $n \ge 1$, where $Q$ is a homogeneous ideal in a polynomial ring $S$, can be any convergent numerical function. 
We also give a positive answer to a long-standing open question of Ratliff on the associated primes of powers of ideals.
\end{abstract}
\maketitle


\section{Introduction}

Let $S$ be a standard graded algebra over a field $k$. 
For a homogeneous ideal $Q \subseteq S$, we call the function $\depth S/Q^n$, $n \ge 1$ 
the \emph{depth function} of $Q$.
The goal of this paper is to prove the following conjecture of Herzog and Hibi in \cite{HH} (see also \cite[Problem 3.10]{He}).
\medskip

\begin{conjecture}[Herzog-Hibi] \label{conj.HH}
Let $f: \NN \rightarrow \ZZ_{\ge 0}$ be any function such that $f(n) = f(n+1)$ for all $n \gg 0$. Then there exists a homogeneous ideal $Q$ in a polynomial ring $S$ such that $f$ is the depth function of $Q$.
\end{conjecture}

For simplicity we call a function $f: \NN \rightarrow \ZZ_{\ge 0}$ a \emph{numerical function} 
and say that $f$ is \emph{convergent} if $f(n) = f(n+1)$ for all $n \gg 0$. By a classical result of Brodmann \cite{Br}, the depth function of a homogeneous ideal is always convergent. Conjecture \ref{conj.HH} simply says that this is the only constraint for numerical functions to be depth functions of homogeneous ideals. This conjecture is remarkable since the depth function tends to be non-increasing in known examples.   

Before this work, Conjecture \ref{conj.HH} has been verified only for non-decreasing functions \cite{HH} 
and for some special classes of non-increasing functions \cite{HTT, HH, MST}. 
Note that the proof of Conjecture \ref{conj.HH} for non-increasing functions in \cite{HTT} has a gap.   Examples of non-monotone depth functions were hard to find \cite{BHH, HS, HH, MV}.
However, Bandari, Herzog and Hibi \cite{BHH} showed that the depth function can have any given number of local maxima.  \par

Our main result, Theorem \ref{Herzog-Hibi}, settles Conjecture \ref{conj.HH} in its full generality. 
Furthermore, we shall show that the ideal $Q$ can be chosen to be a monomial ideal.
As a consequence, we give a positive answer to the following question of Ratliff, 
which has remained open since 1983 \cite[(8.9)]{Ra}.
\medskip

\begin{question}[Ratliff] \label{question.R}
Given a finite set $\Gamma$ of positive integers, do there exist a Noetherian
ring $S$, an ideal $Q$ and a prime ideal $P \supset Q$ in $S$ such that $P$ is an associated prime of $Q^n$ if and only if $n \in \Gamma$?
\end{question}

Inspired by Theorem \ref{Herzog-Hibi}, one may expect that for any convergent positive numerical function $f$, there exists a homogeneous ideal $Q$ such that $f$ is the depth function of symbolic powers of $Q$. This is verified recently by the second and the third authors of this paper \cite{NgT}.

The proof of Conjecture \ref{conj.HH} is based on our recent works on sums of ideals \cite{HNTT, HTT}. The key observation is the additivity of depth functions; that is, the sum of two depth functions is again a depth function. It can also be seen that any convergent numerical function which is not the constant zero function can be written as the sum of a finite number of functions of the following two types:
\begin{itemize}
	\item Type I: for some fixed $d \in \NN$, $f(n) = \left\{\begin{array}{ll} 0 & \text{if } n < d \\ 1 & \text{if } n \ge d. \end{array}\right.$
	\item Type II: for some fixed $d \in \NN$, $f(n) = \left\{\begin{array}{ll} 0 & \text{if } n \not= d \\ 1 & \text{if } n = d. \end{array}\right.$
\end{itemize}
Therefore, the proof is completed if we can construct ideals with depth functions of Types I and II.

Our paper is structured as follows. In Section 2 we prove the additivity of depth functions. Ideals with depth functions of Types I and II are constructed in Section 3. Section 4 is devoted to consequences of our solution to Conjecture \ref{conj.HH}.

We assume that the reader is familiar with basic properties of associated primes and depth, which we use without references. For unexplained notions and terminology, we refer the reader to \cite{BrH, E}.


\section{Additivity of depth functions} \label{sect_prel}

Throughout this section, let $A$ and $B$ be polynomial rings over a field $k$ with disjoint sets of variables, and let $R = A \otimes_k B$. Let $I \subseteq A$ and $J \subseteq B$ be nonzero proper homogeneous ideals. By abuse of notations, we shall also use $I$ and $J$ to denote their extensions in $R$. 

\begin{lemma}[\protect{\cite[Lemma 1.1]{HT}}] \label{HoaTam}
	$I \cap J = IJ$.
\end{lemma}

\begin{lemma}[\protect{\cite[Lemmas 2.2]{HT}}] \label{HoaTam2}
	$\depth R/IJ = \depth A/I + \depth B/J + 1.$ 
	\end{lemma}
	
We shall use the above lemmas to prove the following result which yields the additivity of depth functions.

\begin{proposition} \label{additivity}
Let $I \subset A$ and $J \subset B$ be homogeneous ideals as above.
There exists a homogeneous ideal $Q$ in a polynomial ring $S$ such that for all $n > 0$,
$$\depth S/Q^n = \depth A/I^n + \depth B/J^n.$$
Moreover, if $I$ and $J$ are monomial ideals, then $Q$ can be chosen to be a monomial ideal.
\end{proposition}

\begin{proof}
Let $x \in A$ and $y \in B$ be arbitrary variables.
Let $R = A \otimes_k B$. Then $R$ is a polynomial ring in the variables of $A$ and $B$.
By Lemma \ref{HoaTam} we have $IJ = I \cap J$.
The associated primes of $I \cap J$ are extensions of ideals in one of the rings $A,B$.
Therefore, $x-y$ does not belong to any associated prime of $IJ$.
From this it follows that
$$\depth R/(IJ,x-y) = \depth R/IJ -1.$$
By Lemma \ref{HoaTam2} we have
$$\depth R/IJ = \depth A/I + \depth B/J +1.$$
Therefore,
$$\depth R/(IJ,x-y) = \depth A/I + \depth B/J.$$
Obviously, we may replace $I,J$ by $I^n,J^n$ and obtain
$$\depth R/(I^nJ^n,x-y) = \depth A/I^n + \depth B/J^n.$$
Set $S = R/(x-y)$ and $Q = (IJ,x-y)/(x-y)$. Then $S$ is isomorphic to a polynomial ring over $k$
and
$$\depth S/Q^n = \depth R/((IJ)^n,x-y) = \depth A/I^n + \depth B/J^n$$
for all $n > 0$. Moreover, 
$Q$ is a monomial ideal if $I, J$ are monomial ideals.
\end{proof}

To ease on notations, we shall identify a numerical function $f(n)$ with the sequence of its values $f(1),f(2),....$

If $f$ is the constant function 0,0,0,..., 
then $f$ is the depth function of the maximal homogeneous ideal of any polynomial ring over $k$.

\begin{lemma} \label{Types}
Let $f$ be a convergent numerical function which is not the constant confunction 0,0,0,.... 
Then $f$ can be written as a sum of numerical functions of the following two types:\par
{\rm Type I:} $0,...,0,1,1,...$, \par
{\rm Type II:} $0,...,0,1,0,0,...$. \par
\end{lemma}

\begin{proof}
Let $f$ be a convergent numerical function of the form $c_1,...,c_n, c,c,...$. 
Then $f$ is the sum of the functions $0,...,0,c_i,0,0,...$, $i = 1,..,n$, and the functions $0,...,0,c,c,...$. 
The function $0,...,0,c_i,0,0,...$ is $c_i$ times the function $0,...,0,1,0,0,...$, where $1$ stands only at the $i$-th place. The function $0,...,0,c,c,...$ is $c$ times the function $0,...,0,1,1,...$, where $1$ starts from the $(n+1)$-th place.
\end{proof}

By Proposition \ref{additivity} and Lemma \ref{Types}, to establish the validity of Conjecture \ref{conj.HH}, it suffices to construct depth functions of types I and II.

\section{Construction of ideals with depth functions of Types I and II}

Herzog and Hibi \cite{HH} already constructed monomial ideals $I$ whose depth functions can be any non-decreasing convergent numerical function. Therefore, the existence of depth functions of Type I follows from their result.

\begin{example}[\protect{\cite[Theorem 4.1]{HH}}] \label{Type I}
Let $A = k[x,y,z]$. For any integer $d \ge 2$, let 
$I=(x^{d+2},x^{d+1}y,xy^{d+1},y^{d+2},x^dy^2z).$
Then
$$\depth A/I^n = \begin{cases} 0 &\text{if $n \le d-1$},\\ 1 &\text{if $n \ge d$}. \end{cases}$$
\end{example}

We also know that there are monomial ideals $J$ with the depth function $1,...,1,0,0,...$ \cite{HTT, MST}.
The existence of such depth functions can be used to construct depth functions of Type II as follows.

Let $I$ and $J$ be monomial ideals with the depth functions $0,...,0,1,1,...$ and $1,...,1,0,0,...$, where the first 1 of the first function and the last 1 of the second function are on the same place.  
By the proof of Proposition \ref{additivity}, the function
$\depth R/((IJ)^n,x-y)$ is a function of the form $1,...,1,2,1,1,...$ for some variables $x,y$.
If we can find variables $x',y'$ such that $x'-y'$ is a non-zerodivisor in $R/((IJ)^n,x-y)$ for all $n \ge1$,
then 
$$\depth R/((IJ)^n,x-y,x'-y') = \depth R/((IJ)^n,x-y)-1$$ 
is a function of the form $0,...,0,1,0,0,....$
Clearly, we can identify $S = R/(x-y,x'-y')$ with a polynomial ring
and $(IJ,x-y,x'-y')/(x-y,x'-y')$ with a monomial ideal in $S$.
\par

To find such variables $x',y'$ we need to know the associated primes of the ideal $((IJ)^n,x-y)$ for all $n \ge 1$.
For convenience, we denote the set of the associated primes and the set of the minimal associated primes of an ideal $Q$ by $\Ass(Q)$ and $\Min(Q)$, respectively.

\begin{proposition} \label{control}
Let $A$ and $B$ be polynomial rings over a field $k$.
Let $I \subset A$ and $J \subset B$ be nonzero proper homogeneous ideals.
Let $x \in A$ and $y \in B$ be arbitrary variables. Let $R = A\otimes_kB$.
Then
\begin{align*}
& \Ass (IJ,x-y)  = \\
& \{(\pp,x-y)|\ \pp \in \Ass (I)\} \cup \{(\qq,x-y)|\ \qq \in \Ass(J)\} \cup \bigcup_{\begin{subarray}{l} \pp \in \Ass (I), x \in \pp\\ \qq \in \Ass (J), y \in \qq \end{subarray}} \Min(\pp+\qq).
\end{align*}
\end{proposition}

\begin{proof}
Let $P$ be an arbitrary prime of $\Ass(I,x-y)$. Then $P = \pp+(x-y)$ for some $\pp \in \Ass(I)$.
If $J \subseteq P$, we must have $J \subseteq (y) \subset P$.
This implies $J = y^dJ'$ for some ideal $J' \subset B$, $J' \not\subseteq (y)$, $d \ge 1$.
Let $f \in A$ be an element such that $P = (I,x-y):f$.
It is easy to check that $P = (y^dI,x-y) : y^df.$
Hence, $P \in \Ass(y^dI,x-y)$.
Since $(IJ,x-y)_P = (y^dI,x-y)_P$,  this implies $P \in \Ass(IJ,x-y)$.
If $J \not\subseteq P$, we have $(IJ,x-y)_P = (I,x-y)_P$.
Hence, $P \in \Ass(IJ,x-y)$.
So we can conclude that
$$\Ass(I,x-y) = \{(\pp,x-y)|\ \pp \in \Ass (I)\} \subseteq \Ass(IJ,x-y).$$
\par

Similarly, we also have
$$\Ass(J,x-y) = \{(\qq,x-y)|\ \qq \in \Ass(J)\} \subseteq \Ass(IJ,x-y).$$
It remains to prove that if $P$ is a prime ideal of $R$, which does not belong to $\Ass(I,x-y)$ nor $\Ass(J,x-y)$,
then $P \in \Ass(IJ,x-y)$ if and only if $P \in \Min(\pp+\qq)$ for some $\pp \in \Ass (I), x \in \pp$, and $\qq \in \Ass (J), y \in \qq $. \par

Without restriction, we may assume that $(IJ,x-y) \subseteq P$.
Since $P \not\in \Ass(I,x-y)$, we have $\depth(R/(I,x-y))_P \ge 1$.
Since $x-y$ is a non-zerodivisor on $I$, this implies $\depth (R/I)_P \ge 2$.
Similarly, we also have $\depth (R/J)_P \ge 2$.
Note that $P \in \Ass(IJ,x-y)$ if and only if $\depth (R/(IJ,x-y))_P = 0$.
By Lemma \ref{HoaTam} we have $IJ = I \cap J$. Hence, $x-y$ is a non-zerodivisor in $R/IJ$.
From  this it follows that $P \in \Ass(IJ,x-y)$ if and only if $\depth (R/IJ)_P = 1$.
Using the exact sequence
$$0 \to (R/IJ)_P  \to (R/I)_P \oplus (R/J)_P \to (R/I+J)_P \to 0$$
we can deduce that $\depth (R/IJ)_P = 1$ if and only if $\depth (R/I+J)_P = 0$, which means $P \in \Ass(I+J)$.
By \cite[Theorem 2.5]{HNTT}, we have
$$\Ass(I+J) = \bigcup_{\begin{subarray}{l} \pp \in \Ass (I)\\ \qq \in \Ass (J)\end{subarray}} \Min(\pp+\qq).$$
Notice that $\pp +\qq$ is not necessarily a prime ideal (see e.g. \cite[Example 2.3]{HNTT}).

If $P \in \Min(\pp+\qq)$, then $P \cap A = \pp$ and $P \cap B = \qq$ by \cite[Lemma 2.4]{HNTT}. Moreover,
$P$ is a bihomogeneous ideal with respect to the natural bigraded structure of $R = A \otimes_k B$.
In this case, $x-y \in P$ implies $x \in P \cap A = \pp$ and $y \in P \cap B = \qq$.
So we can conclude that $P \in \Ass(IJ,x-y)$ if  and only if $P \in \Min(\pp+\qq)$ for some $\pp \in \Ass (I), x \in \pp$, and $\qq \in \Ass (J), y \in \qq $.
\end{proof}

\begin{remark}
{\rm Since Theorem \ref{control} is of independent interest, 
one may ask whether it is true in a more general setting. 
If $I,J$ are not homogeneous, we can use the same arguments to prove the following general formula:
\begin{align*}
\Ass (IJ,x-y)  = & \{(\pp,x-y)|\ \pp \in \Ass (I)\} \cup \{(\qq,x-y)|\ \qq \in \Ass(J)\} \cup\\
& \bigcup_{\begin{subarray}{l} \pp \in \Ass (I)\\ \qq \in \Ass (J)\end{subarray}} \Min(\pp+\qq) \cap V(x-y),
\end{align*}
where $V(x-y)$ denotes the set of prime ideals containing $x-y$. In this case,
$(IJ,x-y)$ may have an associated prime $P \in \Min(\pp+\qq)$ for some $\pp \in \Ass(I)$ and $\qq \in \Ass(J)$ which do not satisfy the conditions $x \not\in \pp$ and $y \not\in \qq$. 
}
\end{remark}

\begin{example}
{\rm Let $A = \QQ[x,z]$ and $I = (x^2+1,z)$. Let $B = \QQ[y,t]$ and $J = (y^2+1,t)$.
Then $I, J$ are prime ideals, $x \not\in I$ and $y \not\in J$. We have
$$\Min_R(R/I+J) = \{(x^2+1,t,z,x-y),(x^2+1,t,z,x+y)\}.$$
Hence
$$\Ass(IJ,x-y) = \{(x^2+1,z,x-y), (y^2+1,t,x-y), (x^2+1,z,t,x-y)\}.$$
These primes do not contain $x$ and $y$.}
\end{example}

Using Proposition \ref{control} we can give a sufficient condition for the existence of variables $x',y'$ such that $x'-y'$ is a non-zerodivisor in $R/((IJ)^n,x-y)$ for all $n \ge 1$.  

\begin{proposition} \label{reduction}
Let $I$ be a proper monomial ideal in $A = k[x_1,...,x_r]$, $r \ge 3$, such that $x_3,...,x_r \in \sqrt{I}$.
Let $J$ be a proper monomial ideal in $B=k[y_1,...,y_s]$, $s \ge 3$, such that $y_3,...,y_s \in \sqrt{J}$.
Let $R = k[x_1,...,x_r,y_1,...,y_s]$.
Assume that $\depth A/I^n > 0$ or $\depth B/J^n > 0$ for some $n > 0$. Then
$$\depth R/((IJ)^n,x_1-y_1,x_2-y_2) = \depth A/I^n + \depth B/J^n - 1.$$
\end{proposition}

\begin{proof}
By the proof of Proposition \ref{additivity} we have
$$\depth R/((IJ)^n,x_1-y_1) = \depth A/I^n + \depth B/J^n \ge 1.$$
It remains to show that $x_2-y_2$ is a non-zerodivisor in $R/((IJ)^n,x_1-y_1)$.
Assume for the contrary that $x_2-y_2 \in P$ for some associated prime $P$ of $((IJ)^n,x_1-y_1)$. By Proposition \ref{control}, $P = \pp+\qq$ for some $\pp \in \Ass(I^n)$, $x_1 \in \pp$, and $\qq \in \Ass(J^n)$, $y_1 \in \qq$.
Note that $\pp$ and $\qq$ are generated by variables in $A$ and $B$.
Since $x_2-y_2 \in \pp+\qq$, we must have $x_2 \in \pp$ and $y_2 \in \qq$.
The assumption $x_3,...,x_r \in \sqrt{I}$ and $y_3,...,y_s \in \sqrt{J}$ implies $x_3,...,x_r \in \pp$ and $y_3,...,y_s \in \qq$.
Hence, $x_1,...,x_r,y_1,...,y_s \in P$. Therefore, $P = (x_1,...,x_r,y_1,...,y_s)$, which contradicts the fact that
$\depth R/((IJ)^n,x_1-y_1) \ge 1$.
\end{proof}

Now we are going to construct monomial ideals having depth function of Type II.

\begin{example} \label{Type II}
{\rm Let $A  = k[x,y,z]$ and $I=(x^{d+2},x^{d+1}y,xy^{d+1},y^{d+2},x^dy^2z)$, $d \ge 2$.
By Example \ref{Type I} we have
$$\depth A/I^n = \begin{cases} 0 &\text{if $n \le d-1$},\\ 1 &\text{if $n \ge d$}. \end{cases}$$
Let $B = k[t,u,v]$. Let $J$ be the integral closure of the ideal
$(t^{3d+3},tu^{3d+1}v,u^{3d+2}v)^3$ or $J = (t^{d+1},tu^{d-1}v,u^dv)$.
By \cite[Example 4.10]{HTT} and \cite[Proposition 1.5]{MST} we have
$$\depth B/J^n=\begin{cases} 1 & \text{if $n\le d$},\\ 0 &\text{if $n\ge d+1$}. \end{cases}$$
Let $R = k[x,y,z,t,u,v]$. 
By Proposition \ref{reduction}, we have
$$\depth R/((IJ)^n,y-u,z-v) = \begin{cases} 0 &\text{if $n\neq d$},\\ 1 &\text{if $n = d$}. \end{cases}$$
If we set $S = k[x,t,u,v]$ and $Q = (x^{d+2},x^{d+1}u,xu^{d+1},u^{d+2},x^du^2v)J$, which is obtained from $IJ$ by setting $y = u$ and $z = v$, then
$$\depth S/Q^n = \depth R/((IJ)^n,y-u,z-v).$$
Hence, the depth function of $Q$ is of Type II.}
\end{example}

\section{Consequences}

By Examples \ref{Type I} and \ref{Type II} we have monomial ideals with depth functions of Types I and II.
Therefore, the solution to Conjecture \ref{conj.HH} immediately follows from Proposition \ref{additivity} and Lemma \ref{Types}. 

\begin{theorem} \label{Herzog-Hibi}
Let $f$ be any convergent numerical function.
There exists a monomial ideal $Q$ in a polynomial ring $S$ such that $\depth S/Q^n$ $= f(n)$ for all $n \ge 1$.
\end{theorem}

Theorem \ref{Herzog-Hibi} has the following interesting consequence on the associated primes of powers of ideals, which 
gives a positive answer to Question \ref{question.R} of Ratliff.

\begin{corollary} \label{Ratliff}
Let $\Gamma$ be a set of positive integers which is either finite or contains all sufficiently large integers.
Then there exists a monomial ideal $Q$ in a polynomial ring $S$ such that $\mm \in \Ass(Q^n)$ if and only if $n \in \Gamma$, where $\mm$ is the maximal homogeneous ideal of $S$.
\end{corollary}

\begin{proof}
Let $f$ be any convergent numerical function such that $f(n) = 0$ if and only if $n \in \Gamma$.
Then there exists a monomial ideal $Q$ in a polynomial ring $S$ such that $\depth S/Q^n = f(n)$ for all $n \ge 1$.
This is the desired ideal because $\mm \in \Ass(Q^n)$ if and only if $\depth S/Q^n = 0$.
\end{proof}

Corollary \ref{Ratliff} also gives a negative answer to the following question of Ratliff \cite[(8.4)]{Ra}.

\begin{question}[Ratliff] \label{Ratliff 2}
Let $Q$ be an arbitrary ideal in $Q$ in a Noetherian ring $S$. Let $P \supset Q$ be a prime ideal such that $P \in \Ass(Q^m)$ for some $m \ge 1$ and $P \in \Ass(Q^n)$ for all $n$ sufficiently large. Is $P \in \Ass(Q^n)$ for all $n \ge m$?
\end{question}

This question was already answered in the negative by Huckaba \cite[Example 1.1]{Hu}. However, the ideal $Q$ in his example is not a monomial ideal as in the proof of Corollary \ref{Ratliff}. 

One may also ask about the possible function of the projective dimension of powers of a homogeneous ideal.
Let $Q$ be an arbitrary homogeneous ideal in a polynomial ring $S$.
By the Auslander-Buchsbaum formula we have
$$\pd Q^n = \dim S - \depth S/Q^n - 1.$$
Since $\depth S/Q^n$ is a convergent numerical function \cite{Br}, $\pd Q^n$ is also a convergent numerical function.

\begin{corollary} \label{pd}
Let $g$ be an arbitrary convergent numerical function.
There exist a homogeneous ideal $Q$ and a number $c$ such that $\pd Q^n = g(n) + c$ for all $n \ge 1$.
\end{corollary}

\begin{proof}
Let $m = \max_{n \ge 1}g(n)$. Then $f(n) = m - g(n)$, $n \ge 1$, is a convergent numerical function.
By Theorem \ref{Herzog-Hibi}, there exists a
homogeneous ideal $Q$ in a polynomial ring $S$ such that $\depth S/Q^n = f(n)$ for all $n \ge 1$.
Let $d$ be the number of variables of $S$. Set $c = d-m-1$.
Then 
$$\pd Q^n = d - f(n) - 1= d - m + g(n) - 1 = g(n) + c$$
for all $n \ge 1$. 
\end{proof}

It is of interest to know the smallest possible number $c$ for a given function $g$ in Corollary \ref{pd}.
This number is determined by the smallest number of variables of a polynomial ring  which contains a homogeneous ideal with a given depth function $f$. We are not able to compute this number. 
The proof of Theorem \ref{Herzog-Hibi} uses a high number of variables compared to the values of $f$. 
\medskip

\begin{acknowledgement}
H.T. H\`a is partially supported by the Simons Foundation (grant \#279786) and Louisiana Board of Regents (grant \#LEQSF(2017-19)-ENH-TR-25). Hop D. Nguyen and T.N. Trung are partially supported by Project ICRTM 01$\_$2019.01 of the International Centre for Research and Postgraduate Training in Mathematics. 
N.V. Trung is partially supported by Vietnam National Foundation for Science and Technology Development.
The authors would like to thank Aldo Conca and J\"urgen Herzog for useful discussions, Takayuki Hibi for pointing out a gap of the proof of Conjecture \ref{conj.HH} for non-increasing functions in \cite{HTT}, and C\u{a}t\u{a}lin Ciuperc\u{a} for informing that our negative answer to  Question \ref{Ratliff 2} of Ratliff was already given by S. Huckaba in \cite{Hu}.
This paper is split from the first version of \cite{HNTT} following a recommendation of its referee.
\end{acknowledgement}


\end{document}